\pdfoutput=1
\documentclass[11pt,a4paper]{amsart}
\usepackage[margin=1in]{geometry}
\usepackage[utf8]{inputenc}
\usepackage[T1]{fontenc}
\usepackage{textcase}
\usepackage[dvipsnames]{xcolor}
\usepackage{microtype}
\usepackage{fnpct}

\usepackage{amsmath}
\usepackage{amssymb}
\usepackage{eucal}
\usepackage{mathrsfs}
\usepackage{tikz-cd}

\usepackage{enumitem}
\usepackage{booktabs}
\usepackage[pdfusetitle,colorlinks]{hyperref}
\hypersetup{bookmarksdepth=2,pdfencoding=unicode,allcolors=MidnightBlue}
\usepackage[capitalise,noabbrev,nosort]{cleveref}

\crefname{equation}{}{}
\crefformat{enumi}{(#2#1#3)}
\crefname{enumi}{}{}
\newlist{conenum}{enumerate}{1}
\setlist[conenum,1]{label=(\roman*),ref=\roman*}
\creflabelformat{conenumi}{(#2#1#3)}
\crefname{conenumi}{}{}

\numberwithin{equation}{section}
\theoremstyle{plain}
\newtheorem{Theorem}{Theorem}

\crefname{Theorem}{Theorem}{Theorems}
\newtheorem{theorem}[equation]{Theorem}
\newtheorem{proposition}[equation]{Proposition}
\newtheorem{lemma}[equation]{Lemma}
\newtheorem{corollary}[equation]{Corollary}

\theoremstyle{definition}
\newtheorem{definition}[equation]{Definition}
\newtheorem{example}[equation]{Example}
\newtheorem{question}[equation]{Question}

\theoremstyle{remark}
\newtheorem{remark}[equation]{Remark}

\newcommand{\FF}{\mathbf{F}}
\newcommand{\ZZ}{\mathbf{Z}}
\newcommand{\RR}{\mathbf{R}}

\newcommand{\E}{\mathrm{E}}

\newcommand{\D}{\operatorname{D}}
\newcommand{\Ext}{\operatorname{Ext}}

\newcommand{\Mod}{\operatorname{Mod}}
\newcommand{\Shv}{\operatorname{Shv}}
\newcommand{\Spec}{\operatorname{Spec}}
\newcommand{\Sym}{\operatorname{Sym}}
\newcommand{\Tot}{\operatorname{Tot}}
\newcommand{\V}{\operatorname{V}}
\newcommand{\cofib}{\operatorname{cofib}}
\newcommand{\coker}{\operatorname{coker}}
\newcommand{\fib}{\operatorname{fib}}
\newcommand{\id}{\operatorname{id}}
\newcommand{\op}{\operatorname{op}}

\newcommand{\X}{\mathord{-}}

\newcommand{\Cat}[1]{\mathsf{#1}}
\newcommand{\Cls}[1]{\mathscr{#1}}

\title{On cohomology of locally profinite sets}
\author{Ko Aoki}
\address{Max Planck Institute for Mathematics,
  Vivatsgasse 7, 53111 Bonn, Germany
}
\email{aoki@mpim-bonn.mpg.de}
\date{\today}

\begin{document}

\begin{abstract}
  We construct a locally profinite set
  of cardinality~\(\aleph_{\omega}\)
  with infinitely many first cohomology classes
  of which any distinct finite product does not vanish.
  Building on this,
  we construct the first example
  of a nondescendable faithfully flat map between commutative rings
  of cardinality~\(\aleph_{\omega}\)
  within Zermelo--Fraenkel set theory.
\end{abstract}

\maketitle

\section{Introduction}\label{s:intro}

Profinite sets,
aka totally disconnected compact Hausdorff spaces,
are cohomologically simple;
their (sheaf-theoretic) cohomological dimension
is zero.
We then consider the notion of a locally profinite set,
i.e., a topological space obtained by removing one point
from a profinite set.
Then the situation becomes significantly more complicated:
Pierce~\cite[Theorem~5.1]{Pierce67} showed that
\(\{0,1\}^{\aleph_{n}}\) minus a point
has cohomological dimension~\(n\)
for any \(n\geq0\).
Wiegand~\cite[Example~3]{Wiegand68}
pointed out that there is a locally profinite set
of cardinality~\(\aleph_{1}\)
with nontrivial first \(\FF_{2}\)-cohomology.
In this paper,
we show the following:

\begin{Theorem}\label{main_w}
  For \(n\geq0\),
  there is a locally profinite set~\(U\)
  of cardinality~\(\aleph_{n}\)
  with \(H^{n}(U;\FF_{2})\neq0\).
\end{Theorem}

Via Stone duality,
the category of profinite sets
is dual to the category of Boolean rings,
i.e., rings in which all elements are idempotent.
This correspondence is powerful;
e.g., Pierce used this to prove the aforementioned result.
Our proof of \cref{main_w} is purely topological,
but we can deduce the following result in homological algebra:

\begin{Theorem}\label{main_si}
  For \(n\geq0\),
  there is a Boolean ring
  of cardinality~\(\aleph_{n}\)
  with self-injective dimension~\(n+1\).
  There is also a Boolean ring of cardinality~\(\aleph_{\omega}\)
  with infinite self-injective dimension.
\end{Theorem}

We then address
the question of whether the product structure of cohomology can also be nontrivial.
We prove that it is indeed the case:

\begin{Theorem}\label{main}
  For \(n\geq0\),
  there is a locally profinite set~\(U\) of cardinality~\(\aleph_{2n+1}\)
  with
  classes \(\eta_{0}\), \dots, \(\eta_{n}\in H^{1}(U;\FF_{2})\)
  whose product \(\eta_{0}\dotsb\eta_{n}\in H^{n+1}(U;\FF_{2})\) is nonzero.
\end{Theorem}

We can produce an interesting counterexample
in commutative algebra from~\cref{main}.
The notion of descendability
was introduced by Mathew~\cite{Mathew16}
(and independently by Balmer~\cite{Balmer16} under a different name).
Classically,
we ask for descent of module spectra.
We can go one level higher\footnote{In fact,
  this is the highest level of descent we can ask for;
  when we consider descent
  of linear presentable \((\infty,n)\)-categories
  (in the sense of~\cite{Stefanich})
  for \(n\geq1\),
  we obtain the same notion.
} and ask for descent
of linear (presentable \(\infty\)-)categories.
This is the notion of descendability.
Unarguably,
the most well-known descent
of modules
is Grothendieck's faithfully flat descent~\cite[Exposé~VIII]{SGA1}
and
it is natural to ask whether this can be upgraded
to linear categories.
It is known that
every faithfully flat map over an \(\E_{\infty}\)-ring~\(A\)
with \(\lvert\pi_{*}A\rvert<\aleph_{\omega}\)
is descendable;
see~\cite[Proposition~3.32]{Mathew16}.
Therefore, constructing a counterexample is challenging,
since algebro-geometric counterexamples,
even those not of finite type,
are typically of countable presentation.
Nevertheless,
\cref{main} gives us an optimal counterexample:

\begin{Theorem}\label{main_des}
  There is a faithfully flat map
  \(B\to B'\)
  between Boolean rings
  of cardinality~\(\aleph_{\omega}\)
  that is not descendable.
\end{Theorem}

\begin{remark}\label{xb0wxc}
While this paper was being written,
  Zelich~\cite{Zelich2} independently\footnote{However, note that
    the first version of~\cite{Zelich2}
    was the starting point of this paper.
  } constructed
  an example of a nondescendable faithfully flat ring map.
  The example there is more of an algebro-geometric flavor;
  the source ring
  is the polynomial ring with countably many variables
  over any algebraically closed field
  of cardinality~\(\beth_{\omega}\).
\end{remark}

Since our counterexample is Boolean,
we can deduce another negative implication
in Clausen--Scholze's condensed mathematics~\cite{Condensed}.
For an infinite cardinal~\(\kappa\),
we write \(\Cat{PFin}_{\kappa}\) 
the category of profinite sets of weight\footnote{The \emph{weight} of a topological space
  is
  the smallest cardinality of a basis.
}~\(<\kappa\).
A \(\kappa\)-condensed object is a hypercomplete sheaf
on~\(\Cat{PFin}_{\kappa}\)
with topology generated by finite jointly surjective families.
When \(\kappa\) is a strong limit cardinal
such as~\(\beth_{\omega}\),
we can describe it in terms
of extremally disconnected sets instead,
and hence it is much more tractable.
However,
in~\cite{AnaSta},
they favored the choice \(\kappa=\aleph_{1}\)
and called it light condensed mathematics.
One reason is that it suffices for everyday mathematical practice
(which this paper does not exemplify),
but there are nonpragmatic reasons as well;
one of them is the Betti stack construction:
In~\cite[Lecture~19]{AnaSta},
they constructed a functor
from light condensed animas
to (light) analytic stacks.
The key point is that
\begin{equation}
  \label{e:sg98b}
  \Mod_{\Shv(\X;\D(\ZZ))}(\Cat{Pr})
  \colon\Cat{PFin}_{\aleph_{1}}^{\op}\to\Cat{CAT}
\end{equation}
is a hypersheaf,
where \(\Cat{Pr}\)
and \(\Cat{CAT}\)
denote the large \(\infty\)-category
of presentable \(\infty\)-categories
and the very large \(\infty\)-category
of large \(\infty\)-categories,
respectively.
The proof relies on the lightness of the setting,
so it is natural to ask whether it fails in the nonlight case.
Our counterexample,
which is optimal,
shows that the smallness of~\(\kappa\) is indeed crucial:

\begin{Theorem}\label{main_b}
  The functor
  \(
    \Mod_{\Shv(\X;\D(\FF_{2}))}(\Cat{Pr})
    \colon\Cat{PFin}_{\aleph_{\omega+1}}^{\op}\to\Cat{CAT}
  \)
  does not satisfy (Čech) descent.
\end{Theorem}

Finally,
we comment on our generalities
in this paper:

\begin{remark}\label{x2z0rb}
  By a slight modification of our arguments,
  we can replace our coefficients~\(\FF_{2}\) with any nontrivial ring
  in \cref{main_w,main}.
  As a consequence,
  for any prime~\(p\),
  we can use \(p\)-Boolean rings
  in \cref{main_si,main_des}
  and replace~\(\FF_{2}\) with~\(\FF_{p}\) in \cref{main_b}.
  However,
  to keep the exposition simple,
  we stick to the case of~\(\FF_{2}\).
\end{remark}

\subsection*{Organization}\label{ss:outline}

In \cref{s:main},
we study topology
and prove \cref{main_w,main}.
In \cref{s:app},
we explain how to deduce \cref{main_si,main_des,main_b}
from our counterexamples.

\subsection*{Acknowledgments}\label{ss:ack}

In the first version of~\cite{Zelich2},
Zelich attempted to construct an example
of a faithfully flat map that is not descendable
in the Boolean setting.
Although it turned out to contain mistakes,
it motivated me to study the subject of this paper.

I thank
Benjamin Antieau,
Robert Burklund,
Ishan Levy,
Peter Scholze,
and
Germán Stefanich
for
helpful discussions.
I thank Scholze for useful comments on a draft.
I thank the Max Planck Institute for Mathematics
for its financial support.

\section{Cohomology of locally profinite sets}\label{s:main}

We construct a space and a class
for \cref{main_w}
in \cref{ss:plank}
and prove that it is nonzero in \cref{ss:higher}.
Then we prove \cref{main} in \cref{ss:nn}.

\subsection{Construction}\label{ss:plank}

\begin{definition}\label{x8x476}
  For a nonempty sequence of ordinals \(\alpha=(\alpha_{0},\ldots,\alpha_{n})\),
  we write~\(X(\alpha)\)
  for the product
  \((\aleph_{\alpha_{0}})^{+}\times\dotsb\times(\aleph_{\alpha_{n}})^{+}\),
  where \((\X)^{+}\)
  denotes the one-point compactification of the discrete topological space.
  We write \(X(\alpha)^{-}\)
  for the space obtained by removing the point at infinity.
\end{definition}

\begin{remark}\label{xy2pou}
  In \cref{x8x476},
  both the cardinality and weight of~\(X(\alpha)\)
  are \(\aleph_{\max_{i}\alpha_{i}}\).
\end{remark}

\begin{example}\label{x2udh4}
Let \(\beta\) be the ordinal satisfying \(2^{\aleph_{0}}=\aleph_{\beta}\).
  The space \(X(0,\beta)\) is commonly called the \emph{Thomas plank},
  since it was a building block
  in the construction of Thomas~\cite{Thomas69}.
\end{example}

\begin{example}\label{xm122r}
Wiegand~\cite[Example~3]{Wiegand68}
  noted that \(X(0,1)^{-}\)
  has nontrivial first \(\FF_{2}\)-cohomology.
  More generally,
  he considered the space \(X(0,\ldots,n)^{-}\)
  for \(n\geq0\)
  in the proof of~\cite[Theorem~4.1]{Wiegand69I}.
  In \cref{xqxs5v},
  we prove that its \(n\)th \(\FF_{2}\)-cohomology
  is nontrivial.
\end{example}

The cohomology of~\(X(\alpha)^{-}\)
has a simple combinatorial description:

\begin{proposition}\label{x22ju6}
  Let \(\alpha=(\alpha_{0},\ldots,\alpha_{n})\)
  be a nonempty sequence of ordinals.
  We write \(U_{i}\subset X(\alpha)^{-}\) for the subset consisting of the points
  whose \(i\)th coordinate is not the point at infinity.
  Then the \(\FF_{2}\)-cohomology of~\(X(\alpha)^{-}\)
  is computed as the oriented Čech complex
  whose \(k\)th term is
  \begin{equation*}
    \prod_{0\leq i_{0}<\dotsb<i_{k}\leq n}
    \Cls{C}(U_{i_{0}}\cap\dotsb\cap U_{i_{k}};\FF_{2}),
  \end{equation*}
  where we write~\(\Cls{C}\) for the ring of continuous functions.
  In particular,
  \(H^{*}(X(\alpha)^{-};\FF_{2})=0\)
  for \({*}>n\).
  Moreover,
  an \(n\)-cocycle is a function
  \(f\colon\aleph_{\alpha_{0}}\times\dotsb\times\aleph_{\alpha_{n}}\to\FF_{2}\),
  and it is a coboundary
  if and only if
  it can be written as a sum \(f_{0}+\dotsb+f_{n}\)
  with~\(f_{i}\) almost constant in the \(i\)th coordinate;
  i.e.,
  for any \((x_{0},\ldots,x_{i-1},x_{i+1},\ldots,x_{n})
  \in\aleph_{\alpha_{0}}\times\dotsb\times\aleph_{\alpha_{i-1}}
  \times\aleph_{\alpha_{i+1}}\times\dotsb\times\aleph_{\alpha_{n}}\),
  the function
  \(x_{i}\mapsto f_{i}(x_{0},\ldots,x_{n})\)
  is constant up to a finite set.
\end{proposition}

\begin{proof}
  We see that
  \(U_{i_{0}}\cap\dotsb\cap U_{i_{k}}\)
  for \(0\leq i_{0}<\dotsb<i_{k}\leq n\)
  is a disjoint union
  of profinite sets,
  and hence it does not have higher cohomology.
  The rest is immediate.
\end{proof}

\begin{remark}\label{xcxc8o}
  As shown in \cref{x22ju6},
  we can concretely study the highest cohomology group.
  This is the advantage of this construction
  in contrast to the approach of Pierce~\cite[Section~5]{Pierce67}.
\end{remark}

We then construct a typical top cohomology class:

\begin{definition}\label{xayddy}
  We fix a nonempty sequence of ordinals
  \(\alpha=(\alpha_{0},\ldots,\alpha_{n})\)
  and a decomposition
  \(\aleph_{\alpha_{i}}=E_{i}\amalg F_{i}\)
  with \(\lvert E_{i}\rvert=\lvert F_{i}\rvert\)
  for each~\(i\).
  We consider the function
  \(e\colon\aleph_{\alpha_{0}}\times\dotsb\times\aleph_{\alpha_{n}}\to\FF_{2}\)
  given by
  \begin{equation*}
    e(x_{0},\ldots,x_{n})
    =
    \begin{cases}
      1&\text{if \(x_{i}\in E_{i}\) for all~\(i\),}\\
      0&\text{otherwise.}
    \end{cases}
  \end{equation*}
  We write~\(\epsilon\) for the class
  in \(H^{n}(X(\alpha)^{-};\FF_{2})\)
  corresponding to~\(e\) via \cref{x22ju6}.
\end{definition}

\subsection{Nontriviality}\label{ss:higher}

\Cref{main_w} follows from the following:

\begin{theorem}\label{xqxs5v}
  The class~\(\epsilon\in H^{n}(X(\alpha)^{-};\FF_{2})\)
  defined in \cref{xayddy}
  is nonzero
  when \(\alpha_{i}<\alpha_{i+1}\)
  for \(0\leq i<n\).
\end{theorem}

\begin{remark}\label{xgwz1r}
The case where \(\alpha_{i}=i\) in \cref{xqxs5v}
  implies the generalization of \cref{xqxs5v}
  in which only \(\alpha_{i}\geq i\) is assumed;
  e.g., \(\epsilon\) is nonzero for the constant sequence \(\alpha_{i}=n\).
  However, this asymmetry makes the proof clearer.
\end{remark}

The following trivial observation is the key in the proof:

\begin{lemma}\label{pq}
  Let \(P\) and \(Q\) be infinite sets
  and \(f\colon P\times Q\to\FF_{2}\)
  a function
  that is almost constant in the second factor.
  Then there is \(Q'\subset Q\)
  with \(\lvert Q\setminus Q'\rvert\leq\lvert P\rvert\)
  such that \(f\restriction P\times Q'\)
  is constant in the second factor.
\end{lemma}

\begin{proof}
  For each \(p\in P\),
  take a finite subset~\(Q_{p}\subset Q\)
  such that \(f\restriction\{p\}\times(Q\setminus Q_{p})\)
  is constant.
  Then \(Q'=Q\setminus\bigcup_{p\in P}Q_{p}\) satisfies the condition.
\end{proof}

\begin{proof}[Proof of \cref{xqxs5v}]
  We proceed by induction on~\(n\).
  The statement is trivial for \(n=0\).
  We consider the case \(n>0\).
  We assume that \(e\) is a coboundary for a contradiction.
  Then we can write~\(e\) as \(e_{0}+\dotsb+e_{n}\)
  where each~\(e_{i}\) is almost constant in the \(i\)th coordinate.
  We apply \cref{pq} to~\(f=e_{n}\),
  \(P=\aleph_{\alpha_{0}}\times\dotsb\times\aleph_{\alpha_{n-1}}\),
  and \(Q=\aleph_{\alpha_{n}}=E_{n}\amalg F_{n}\).
  We then pick \(s\in E_{n}\cap Q'\) and \(t\in F_{n}\cap Q'\).
  For \(f\colon\aleph_{\alpha_{0}}\times\dotsb\times\aleph_{\alpha_{n}}\to\FF_{2}\),
  we define \(f'\colon\aleph_{\alpha_{0}}\times\dotsb\times\aleph_{\alpha_{n-1}}\to\FF_{2}\)
  by
  \begin{equation*}
    f'(x_{0},\ldots,x_{n-1})
    =
    f(x_{0},\ldots,x_{n-1},s)-f(x_{0},\ldots,x_{n-1},t).
  \end{equation*}
  Then we have \(e'=e_{0}'+\dotsb+e_{n-1}'\),
  since \(e_{n}'\) is zero
  by the definition of~\(s\) and~\(t\).
  This means that
  the \((n-1)\)-cocycle~\(e'\)
  on \(X(\alpha_{0},\ldots,\alpha_{n-1})\)
  is a coboundary,
  which contradicts our inductive hypothesis.
\end{proof}

\subsection{Non-nilpotence}\label{ss:nn}

We here use \cref{xqxs5v} to prove \cref{main}.
The main difficulty is to control the product structure.
The key observation is the following,
which is a pathological variant of
the isomorphism \(H^{n+1}((\RR^{2}\setminus\{0\})^{n+1})
\simeq H^{2n+1}(\RR^{2n+2}\setminus\{0\})\):

\begin{proposition}\label{x2gquz}
For a sequence of ordinals
  \(\alpha=(\alpha_{0},\ldots,\alpha_{2n+1})\),
  we write~\(U\) for the product
  \(X(\alpha_{0},\alpha_{1})^{-}\times\dotsb\times X(\alpha_{2n},\alpha_{2n+1})^{-}\).
  Then \(H^{*}(U;\FF_{2})=0\) for \({*}>n+1\),
  and we have a canonical isomorphism
  \(H^{n+1}(U;\FF_{2})\simeq H^{2n+1}(X(\alpha)^{-};\FF_{2})\).
\end{proposition}

\begin{proof}
We write~\(U\)
  for the locally profinite set on the right-hand side.
  For \(U_{i}=X(\alpha_{2i},\alpha_{2i+1})^{-}\),
  we write \(V_{i}\) and~\(W_{i}\)
  for the subset
  consisting of the points
  whose \(2i\)th and \((2i+1)\)st coordinates are not the points at infinity,
  respectively.
  Then the constant sheaf~\(\FF_{2}\in\Shv(U_{i};\D(\FF_{2}))\) is equivalent to
  \begin{equation*}
    \cofib\bigl((\FF_{2})_{V_{i}\cap W_{i}}\to(\FF_{2})_{V_{i}}\oplus(\FF_{2})_{W_{i}}\bigr);
  \end{equation*}
  cf.~\cref{x22ju6}.
  Therefore, the constant sheaf \(\FF_{2}\in\Shv(U;\D(\FF_{2}))\) is equivalent to
  \begin{equation*}
    \cofib\bigl((\FF_{2})_{V_{0}\cap W_{0}}
    \to(\FF_{2})_{V_{0}}\oplus(\FF_{2})_{W_{0}}\bigr)
    \boxtimes\dotsb\boxtimes
    \cofib\bigl((\FF_{2})_{V_{n}\cap W_{n}}
    \to(\FF_{2})_{V_{n}}\oplus(\FF_{2})_{W_{n}}\bigr).
  \end{equation*}
  By comparing its mapping spectrum against~\(\FF_{2}\) with \cref{x22ju6},
  we get the desired result.
\end{proof}

We then obtain the precise version of \cref{main}:

\begin{theorem}\label{x728ui}
  For \(n\geq0\),
  we consider \(U=X(0,1)^{-}\times\dotsb\times X(2n,2n+1)^{-}\).
  We write \(\eta_{i}\) for
  the pullback of the class in \(H^{1}(X(2i,2i+1)^{-};\FF_{2})\)
  constructed in \cref{xayddy}.
  Then \(\eta_{0}\dotsb\eta_{n}\) is nonzero.
\end{theorem}

\begin{proof}
  With the isomorphism in \cref{x2gquz},
  the product corresponds to
  the class constructed in \cref{xayddy}
  for \((0,\ldots,2n+1)\),
  which is nonzero by \cref{xqxs5v}.
\end{proof}

\begin{remark}\label{xz6mfb}
  In \cref{x728ui},
  we can take \(U=(X(2n+1,2n+1)^{-})^{n}\)
  to obtain a more symmetric example;
  see \cref{xgwz1r}.
\end{remark}

We conclude this section with a question:

\begin{question}\label{xci44g}
  Is \cref{main} optimal in terms of the weight (or cardinality) of~\(U\)?
\end{question}

\section{Applications}\label{s:app}

In this section,
we prove \cref{main_si,main_des,main_b}
in \cref{ss:stone,ss:des,ss:betti}, respectively.

\subsection{Boolean rings can have infinite self-injective dimensions}\label{ss:stone}

We recall the following form of Stone duality:

\begin{theorem}\label{x4mdjv}
  Let \(\Cat{Bool}\) and \(\Cat{PFin}\)
  denote the categories of Boolean rings
  and profinite sets, respectively.
  Then the functors \(\Spec\)
  and \(\Cls{C}(\X;\FF_{2})\)
  define an equivalence
  \(\Cat{Bool}^{\op}\simeq\Cat{PFin}\).
  For a Boolean ring~\(B\),
  the category of \(B\)-modules
  is canonically equivalent
  to that of sheaves of \(\FF_{2}\)-vector spaces
  on \(\Spec B\).
\end{theorem}

\begin{remark}\label{x9rfq1}
  Note that
  for an infinite profinite set~\(X\),
  the weight of~\(X\)
  coincides with \(\lvert\Cls{C}(X;\FF_{2})\rvert\).
\end{remark}

\begin{corollary}\label{xhwhh2}
  Any Boolean ring is absolutely flat;
  i.e., any module is flat.
\end{corollary}

\begin{corollary}\label{xwyfcu}
  A map from a Boolean ring to a ring is faithfully flat
  if and only if it is injective.
\end{corollary}

\begin{corollary}\label{xjlvu9}
  For \(n\geq0\),
  any Boolean ring with cardinality~\(\aleph_{n}\)
  has global dimension~\(\leq n+1\).
\end{corollary}

\begin{proof}
  This follows from \cref{xhwhh2} and~\cite[Corollary~1.4]{Osofsky68}.
\end{proof}

The self-injective dimension of a Boolean ring
has a topological interpretation:

\begin{proposition}\label{xazmzz}
Let \(B\) be a Boolean ring
  and \(X=\Spec B\) the corresponding profinite set.
  Then for \(n\geq0\),
  its self-injective dimension is~\(\leq n+1\)
  if and only if
  \(H^{*}(U;\FF_{2})=0\) for \({*}>n\)
  for any open subset~\(U\) of~\(X\).
\end{proposition}

\begin{proof}
  By the Baer criterion,
  the condition
  on the self-injective dimension
  is equivalent to \(\Ext^{*}_{B}(B/J,B)=0\)
  for \({*}>n+1\)
  and any ideal~\(J\).
  By the exact sequence \(J\to B\to B/J\),
  this is furthermore equivalent
  to \(\Ext^{*}_{B}(J,B)=0\)
  for \({*}>n\).
  However,
  by \cref{x4mdjv},
  \(\Ext^{*}_{B}(J,B)\simeq
  H^{*}(U;\FF_{2})\)
  for \(U=\Spec B\setminus\V(J)\).
\end{proof}

\begin{remark}\label{xddxfq}
  We cannot consider \(n=-1\) in \cref{xazmzz}.
  However,
  it is known that
  a Boolean ring is self-injective
  if and only if
  the corresponding profinite set is extremally disconnected;
  see~\cite[Section~2.4]{Lambek66}.
\end{remark}

\begin{proof}[Proof of \cref{main_si}]
  By \cref{main_w},
  there is a locally profinite set~\(U\)
  of weight~\(\aleph_{n}\)
  with \(H^{n}(U;\FF_{2})\neq0\).
  Then we take \(B=\Cls{C}(U^{+};\FF_{2})\),
  where \(U^{+}\) is the one-point compactification of~\(U\).
  We obtain the desired result by \cref{xazmzz}.
  For the second part,
  we argue in the same way
  starting
  with the disjoint union of the examples in \cref{main_w}
  for each~\(n\).
\end{proof}

\subsection{Faithful flatness does not imply descendability}\label{ss:des}

We recall the following quantitative version of descendability
from~\cite[Definition~11.18]{BhattScholze17}:

\begin{definition}\label{xw6pol}
  We say that
  a map of \(\E_{\infty}\)-rings \(f\colon A\to B\)
  is \emph{descendable of index \(\leq n\)}
  if \(\fib(f)^{\otimes_{A}n}\to A\) is null.
  We say that \(f\) is \emph{descendable} if
  it is descendable of index \(\leq n\) for some \(n\geq0\).
\end{definition}

\begin{example}\label{x9pw44}
  A map \(A\to B\) of \(\E_{\infty}\)-rings
  is descendable of index \(\leq0\)
  if and only if \(A=0\) (hence \(B=0\)).
\end{example}

\begin{example}\label{x5dzhu}
  For maps \(A\to B\to C\) of \(\E_{\infty}\)-rings,
  if \(A\to C\) is descendable of index \(\leq n\),
  so is \(A\to B\).
\end{example}

\begin{example}\label{xxw7dx}
  For \(n\geq0\),
  suppose that the self-injective dimension of
  a ring~\(A\) is \(<n\),
  Then any faithfully flat map \(f\colon A\to B\)
  is descendable of index \(\leq n\),
  since the class
  \(\fib(f)^{\otimes_{A}n}\to A\)
  is an element of \(\Ext^{n}_{A}(\coker(f)^{\otimes_{A}n},A)\),
  which vanishes.
\end{example}

\begin{example}\label{xm9jpg}
  For \(n\geq0\),
  by combining \cref{xxw7dx}
  with \cref{xjlvu9},
  we see that
  a faithfully flat map from a Boolean ring~\(B\)
  is descendable of index \(\leq n+2\)
  when \(\lvert B\rvert\leq\aleph_{n}\).
\end{example}

We recall the following from the first version of~\cite{Zelich2},
whose construction is attributed to Aise~Johan de~Jong there:

\begin{proposition}\label{deformed}
  Consider a ring~\(A\),
  a flat \(A\)-module~\(M\),
  and a class \(\eta\in\Ext^{1}_{A}(M,A)\).
  Then the deformed symmetric algebra \(f\colon A\to B\)
  (see the proof below for the construction)
  is a faithfully flat map of rings
  such that
  the map~\(f\) is not descendable of index \(\leq n\)
  when \(\eta^{n}\in\Ext^{n}_{A}(M^{\otimes_{A}n},A)\) does not vanish.
\end{proposition}

\begin{proof}
Here we give another interpretation of the original construction,
  which I learned through a discussion with Antieau and Stefanich:
  We consider the pushout
  \begin{equation*}
    \begin{tikzcd}
      \Sym_{A}(M[-1])\ar[r]\ar[d]&
      A\ar[d]\\
      A\ar[r,"f"]&
      B
    \end{tikzcd}
  \end{equation*}
  of derived rings,
  where the top and left arrows
  are induced by~\(0\) and \(\eta\in\Ext^{1}_{A}(M,A)\),
  respectively.
  By~\cite[Proposition~2.11]{MathewMondal},
  the top arrow is coconnectively faithfully flat
  in the sense of~\cite[Section~2]{MathewMondal}
  and so is~\(f\).
  By construction,
  we obtain a map \(M\to\coker(f)\)
  such that the class in \(\Ext^{1}_{A}(\coker(f),A)\)
  determined by~\(f\) is the image of~\(\eta\).
\end{proof}

In the Boolean situation, the following variant is useful:

\begin{proposition}\label{xbx6x1}
  In the situation of \cref{deformed},
  assume that \(A\) is Boolean.
  Then the Boolean quotient of~\(B\),
  i.e., the ring obtained by killing \(x^{2}-x\) for all~\(x\),
  is still faithfully flat over~\(A\)
  and has the same property about descendability.
\end{proposition}

\begin{proof}
It suffices to show the faithful flatness,
  since the rest follows from \cref{x5dzhu}.
  We can check this after base change along each point~\(A \to \FF_{2}\) of~\(\Spec A\),
  and hence we may assume that~\(A\) is~\(\FF_{2}\).
  Since \(\eta\) is trivial in this case,
  we have a section \(A\to B\to\FF_{2}\),
  which shows that the Boolean quotient of~\(B\) is nontrivial.
\end{proof}

\begin{theorem}\label{xl75b9}
  For \(n\geq0\),
  there is a faithfully flat map of Boolean rings
  \(B\to B'\) that is not descendable of index \(\leq n\)
  and \(\lvert B'\rvert<\aleph_{2n}\).
\end{theorem}

\begin{proof}
  When \(n=0\),
  we take \({\id}_{\FF_{2}}\).
  When \(n>0\),
  we take an example~\(U\)
  and \((\eta_{0},\ldots,\eta_{n-1})\in H^{1}(U;\FF_{2})^{n}\)
  from \cref{main} for \(n-1\).
  We then consider \(B=\Cls{C}(U^{+};\FF_{2})\)
  and its ideal~\(J\)
  consisting of the functions vanishing at infinity.
  We get the desired counterexample
  by using \(\eta\in\Ext^{1}_{B}(J^{\oplus n},B)\)
  as an input of \cref{xbx6x1}.
\end{proof}

\begin{remark}\label{x6rjs7}
\Cref{xl75b9} is likely not optimal.
  One can ask for an example
  where \(B'\) has smaller cardinality;
  see also \cref{xci44g}.
  Nevertheless, \cref{main_des} is optimal.
\end{remark}

\begin{proof}[Proof of \cref{main_des}]
  By considering
  the disjoint union of the examples in \cref{main}
  for each~\(n\),
  we obtain a locally profinite set~\(U\) of weight \(\aleph_{\omega}\)
  with \((\eta_{i})_{i}\in H^{1}(U;\FF_{2})^{\omega}\)
  such that
  \(\eta_{0}\dotsb\eta_{i-1}\neq0\)
  for any~\(i\).
  Then the rest is the same as our argument for \cref{xl75b9}.
\end{proof}

\subsection{The higher sheaves functor is not (heavy) condensed}\label{ss:betti}

We here note the following,
which implies \cref{main_b}:

\begin{theorem}\label{xyhogu}
  Let \(B\to B'\) be the map of \cref{main_des}
  and \(X_{\bullet}\to X_{-1}\)
  denote the augmented simplicial profinite sets
  obtained as the nerve of \(\Spec B'\to\Spec B\).
Then the functor
  \begin{equation*}
    \Mod_{\Shv(X_{-1};\D(\FF_{2}))}\Cat{Pr}
    \to
    \Tot\bigl(\Mod_{\Shv(X_{\bullet};\D(\FF_{2}))}\Cat{Pr}\bigr)
  \end{equation*}
  is not an equivalence.
\end{theorem}

\begin{proof}
By~\cite[Proposition~3.43]{Mathew16},
  which is attributed to Lurie there,
  the equivalence would imply
  that
  \(B=\Cls{C}(X_{-1};\FF_{2})\to\Cls{C}(X_{0};\FF_{2})=B'\)
  is descendable,
  which is false by assumption.
\end{proof}

We conclude this paper with several remarks:

\begin{remark}\label{xmlb8j}
  Note that \cref{main_b} is about descent
  rather than hyperdescent.
  In fact,
  the mere failure of hyperdescent
  can be deduced only from \cref{main_si}
  with more work.
Note also that
  the fundamental question of
  whether \(\Shv(\Cat{PFin}_{\kappa})\)
  is hypercomplete
  is open for any \(\kappa>\aleph_{0}\).
\end{remark}

\begin{remark}\label{xha10z}
\Cref{main_b} is optimal
  since
  \cref{e:sg98b} is also a hypersheaf
  when we replace~\(\aleph_{1}\)
  with~\(\aleph_{\omega}\);
  cf.~\cref{xm9jpg}.
  Hence this result would not convince the reader
  that \(\aleph_{1}\)-condensed mathematics
  is better than \(\aleph_{\alpha}\)-condensed mathematics
  for any other \(\alpha\in\omega+1\).
  Indeed,
  assuming \(\aleph_{\omega}=\beth_{\omega}\),
  it may be convenient to choose \(\alpha=\omega\),
  since it allows us to use extremally disconnected sets.
\end{remark}

\begin{remark}\label{x8fwp0}
  \Cref{main_b} shows that
  the higher sheaves functor is not condensed,
  but we can still consider
  its condensed approximation.
We plan to study this ``completed'' version
  of the presentable \((\infty,n)\)-category
  of higher sheaves in the future.
\end{remark}

\bibliographystyle{plain}
 \newcommand{\yyyy}[1]{}

\end{document}